\documentclass[11pt,a4paper]{amsart}
\usepackage{amsfonts}
\usepackage{}
\usepackage{amsfonts}
\usepackage{mathrsfs}

\usepackage[bookmarks,colorlinks,
            linkcolor=black,
            anchorcolor=black,
            citecolor=blue]{hyperref}
\usepackage{color}              
\usepackage{indentfirst}        
\usepackage{latexsym}
\usepackage{amsmath,amssymb}    
\usepackage{pstricks}
\usepackage{pst-node}
\usepackage{pst-tree}
\usepackage{pst-plot}
\usepackage{pst-text}
\usepackage{graphicx}
\usepackage{cases}
\usepackage{pifont}
\usepackage{txfonts}
\usepackage[all,knot,poly]{xy}

\usepackage{arydshln}

\allowdisplaybreaks[4]  

\newcommand{\al}{\alpha}
\newcommand{\be}{\beta}
\newcommand{\de}{\delta}
\newcommand{\De}{\Delta}

\newcommand{\Ga}{\Gamma}
\newcommand{\la}{\lambda}

\newcommand{\mg}{\mathfrak{g}}
\newcommand{\q}{\mathbf{q}}

\newcommand{\bN}{\mathbb{N}}
\newcommand{\bQ}{\mathbb{Q}}
\newcommand{\bK}{\mathbb{K}}
\newcommand{\bZ}{\mathbb{Z}}
\newcommand{\di}{\diamond}
\newcommand{\La}{\Lambda}
\newcommand{\ot}{\otimes}
\newcommand{\om}{\omega}
\newcommand{\op}{\oplus}
\newcommand{\si}{\sigma}

\newcommand{\ve}{\varepsilon}
\newcommand{\mO}{\mathcal{O}}

\newcommand{\mb}[1]{\mbox{#1}}

\newcommand{\m}[1]{\mathcal{#1}}
\newcommand{\mi}{\mbox{id}}
\newcommand{\bs}[1]{{\scriptsize\mbox{#1}}}

\newcommand{\ol}[1]{\overline{#1}}
\newcommand{\ull}[1]{\underline{#1}}
\newcommand{\lan}{\langle}
\newcommand{\ran}{\rangle}
\newcommand{\lb}{\left(}
\newcommand{\rb}{\right)}

\newcommand{\rw}{\rightarrow}
\newcommand{\wt}{\mbox{\rm wt}}

\theoremstyle{plain}

\newtheorem{theorem}{Theorem}[section]

\newtheorem{lem}[theorem]{Lemma}

\newtheorem{prop}[theorem]{Proposition}

\theoremstyle{remark}
\newtheorem{rem}[theorem]{Remark}

\theoremstyle{definition}
\newtheorem{defn}[theorem]{Definition}

\title[multi-parameter quantum groups via QQSA]
{Multi-parameter Quantum Groups via Quantum Quasi-symmetric Algebras}
\author[Li]{Yunnan Li}
\address{Department of Mathematics, East China Normal University,
Minhang Campus, Dong Chuan Road 500, Shanghai 200241, PR China}
\email{yunnan814@163.com}

\author[Hu]{Naihong Hu$^\star$}
\address{Department of Mathematics, East China Normal University,
Minhang Campus, Dong Chuan Road 500, Shanghai 200241, PR China}
\email{nhhu@math.ecnu.edu.cn}

\author[Rosso]{Marc Rosso}
\address{Universit\'e Paris Diderot - Paris VII, UFR de Math\'ematiques, Case 7012,
B\^atiment Chevaleret, 75205 Paris Cedex 13, France}
\email{rosso@math.jussieu.fr}

\thanks{$^\star$N.~H.,
corresponding author, supported in part by the NNSFC (Grant No.
 11271131).}


\begin{document}

\begin{abstract}
 It is proved that the entire multi-parameter (small-)quantum groups of symmetrizable Kac-Moody algebras can be realized as certain subquotients of the cotensor Hopf algebras. This is an axiomatic construction. Hopf $2$-cocycle deformations variation for the construction machinery is described, moreover, the integrable irreducible modules can be constructed using this setting, even available for those in the fundamental alcove at root of unity case.
\end{abstract}
\maketitle

\section{Introduction}
\noindent
{\it 1.1.}
One of the central problems in quantum groups theory is to find nice realizations of them. After the original definition using generators and relations due to Drinfel’d-Jimbo, in the last decade of the 20th century, Ringel realized a half of the quantum group of finite type via the Hall algebras \cite{Ring}, and Rosso gave an axiomatic construction for a half of the quantum group using quantum shuffle algebras \cite{Ro1}. The former motivates Lusztig to establish the canonical bases theory \cite{Lus,Lus1}, and the latter re-ignites many mathematicians' interests in Nichols algebras, which is essential to the classification of finite-dimensional pointed Hopf algebras (see the work of Andruskiewitsch and Schneider \cite{AS}). Recently, two breakthroughs have been made in seeking the realization for the entire quantum groups. One is due to Bridgeland \cite{Bri}, in which he realized those quantum groups of ADE type via the Hall algebra of the abelian category of $\mathbb{Z}_2$-graded complexes of quiver representations. The other is due to Fang-Rosso \cite{FR}, where an axiomatic construction approach was proposed in virtue of quantum quasi-symmetric algebras.

\noindent
{\it 1.2.} Fang-Rosso's approach developed the notion of Nichols' cotensor coalgebra and was based on some relevant works. First of all, Loday-Ronco \cite{LR} proved a structure theorem for cofree Hopf algebras (analogous to the Cartier-Milnor-Moore theorem) stating that any cofree Hopf algebra is of the form $U2(R)$, where $R$ is a {\bfseries B}$_\infty$-algebra. Later, Jian-Rosso \cite{JR} gave a systematic construction of braided Hopf algebra structures on braided cofree coalgebras, where quantum multi-brace algebras were introduced as the braided version of {\bfseries B}$_\infty$-algebras. Particular interesting examples are quantum shuffle algebras (\cite{Ro1}) and quantum quasi-shuffle algebras (\cite{JRZ}), deeply related to quantum groups. Recently, Fang-Rosso \cite{FR} further dealt with the bosonization of quantum multi-brace algebras, where the two motivations were: the first is to classify all Hopf algebra structures on the cotensor coalgebra $T_H^c(M)$, where $H$ is a Hopf algebra and $M$ is an $H$-Hopf bimodule. The second is to realize the whole quantum groups as specializations of quantum quasi-symmetric algebras.

\noindent
{\it 1.3.}
Hu-Pei-Rosso \cite{PHR} studied the multi-parameter quantum groups associated to symmetrizable generalized Cartan matrices, together with their representations in  category $\mO^{\q}$. Based on the powerful classification theorem in \cite{FR}, the aim of this paper is to give a construction for the whole multi-parameter quantum groups in spirit of quantum quasi-symmetric algebras (in brief, QQSA). Moreover, any irreducible representation naturally appears as the right coinvariant subspace of the degree one component in the machinery.

\noindent
{\it 1.4.}
The paper is organized as follows. In Section 2, we recall the systematic construction of multi-brace cotensor Hopf algebras due to Fang-Rosso, especially an important example, the QQSAs. Section 3 recalls the multi-parameter quantum groups in \cite{PHR} and provides several useful results. We mention that as a crucial step when using the machinery to establish the injectivity of the corresponding Hopf algebra homomorphism, being distinct from the treatment in \cite{FR}, we need to describe the coradical of multi-parameter quantum groups to overcome the technical difficulty point (note the coradical filtration was used by Chin-Musson \cite{CM} for the quantum groups associated to finite Cartan matrices in literature).
In Section 4, we prove our main theorem, the QQSA realization of multi-parameter quantum groups. We further study the matching variation between Hopf $2$-cocycle deformations both for multi-parameter quantum groups and such construction machinery. In Section 5, we give the QQSA realization of the integrable irreducible representations. It is worthy to mention that the construction machinery as asserted by the third author is available both for the generic case and for the root of unity case when one deals with those weights lying in the fundamental alcove.

 Throughout the paper, we work over an algebraically closed field $\bK$ of characteristic $0$.

\section{Quantum quasi-symmetric algebra structure on cotensor coalgebras}

\noindent{\it 2.1.}
Now we briefly recall the construction of multi-brace cotensor Hopf algebras in \cite{FR}. Suppose $C$ is a coalgebra over field $\bK$, $M, N$ are two $C$-bicomodules, with the left and right comodule structure maps denoted by $\de_L,\de_R$. The cotensor product of $M,N$ is defined as the sub-bicomodule of $M\ot_\bK N$,
$M\Box_CN\triangleq\mb{Ker}(\de_R\ot\mi_N-\mi_M\ot\de_L)$.

\begin{defn}
Given a coalgebra $C$ and a $C$-bicomodule $M$, the \textit{cotensor coalgebra} $T_C^c(M)$ is a graded vector space $T_C^c(M)=\bigoplus_{k\geq0} T_C^c(M)_k$, where $T_C^c(M)_0=C,~T_C^c(M)_1=M,~T_C^c(M)_k$ $=M^{\Box_C^k},~k>1$. It has graded coproduct $\De_c: M^{\Box_C^n}\rw \sum_{i+j=n}M^{\Box_C^i}\ull{\ot} M^{\Box_C^j}$ defined as follows: $\De_c|_{T_C^c(M)_0}=\De_C$, while for any monomial $m_1\ot\cdots\ot m_n,~n\geq1$,
\[\begin{split}
\De_c&(m_1\ot\cdots\ot m_n)={m_1}_{(-1)}\ull{\ot}({m_1}_{(0)}\ot m_2\ot\cdots m_n)\\
&+(m_1 \ot\cdots\ot m_{n-1}\ot {m_n}_{(0)})\ull{\ot}{m_n}_{(1)}+\sum_{i=1}^{n-1}(m_1\ot\cdots\ot m_i)\ull{\ot}(m_{i+1}\ot\cdots\ot m_n),
\end{split}\]
and then linearly extended to the whole $M^{\Box_C^n}$.

Let $\pi:T_C^c(M)\rw C$ and $p:T_C^c(M)\rw M$ be the projection of $T_C^c(M)$ to its $0,1$-homogeneous components respectively. The counit of $T_C^c(M)$ is $\ve_C\pi$.
\end{defn}

$T_C^c(M)$ has the following universal property, which is essential for the consideration of its Hopf algebra structure:
\begin{prop}[\cite{Ni}]\label{uni}
Let $C,D$ be two coalgebras, $M$ is a $C$-bicomodule, $g:D\rw C$ is a coalgebra map, $f:D\rw M$ is a $C$-bicomodule map and $f(\mb{corad(D)})=0$, where the left and right $C$-comodule structure maps of $D$ are respectively $(g\ot\mi)\De_D,(\mi\ot g)\De_D$. There exists a unique coalgebra map $F:D\rw T_C^c(M)$ such that the following two commuting diagrams hold:
\[\xymatrix@=1.5em{D\ar@{->}[r]^-{F}\ar@{->}[dr]_-{g}&T_C^c(M)\ar@{->}[d]^-{\pi}\\
&C},\quad\xymatrix@=1.5em{D\ar@{->}[r]^-{F}\ar@{->}[dr]_-{f}&T_C^c(M)\ar@{->}[d]^-{p}\\
&M}.\]
\end{prop}
Note that the coalgebra map $F$ in Proposition \ref{uni} has the following explicit formula:
\begin{equation}\label{mu}
F=g+\sum_{n\geq1}f^{\ot n}\De_D^{(n-1)},
\end{equation}
where $\De_D^{(0)}=\mi_D,~\De_D^{(n)}=(\mi_D^{\ot n-1}\ot\De_D)\De_D^{(n-1)}$.

\begin{lem}
The coalgebra map $F:D\rw T_C^c(M)$ in Proposition \ref{uni} is also an $H$-bicomodule map.
\end{lem}
\begin{proof}
Note that $g\circ F=\pi$, and the left $H$-comodule structure map of $T_C^c(M)$
\end{proof}

\noindent{2.2.}
In particular, Fang-Rosso in \cite{FR} considered a Hopf bimodule $M$ of a Hopf algebra $H$. For $n\geq1$, $M^{\Box_H^n}$ inherits the tensor product $H$-bimodule structure of $M^{\ot n}$, while its $H$-bicomodule structure is given by the map $\de_L\ot\mi_M^{\ot n-1},~\mi_M^{\ot n-1}\ot\de_R$. Now $M^{\Box_H^n}$ becomes an $H$-Hopf bimodule, so does $T_H^c(M)$. Now let $D=T_H^c(M)\ull{\ot}T_H^c(M)$, endowed with the usual coproduct of tensor products of coalgebras. $D$ is also an $H$-Hopf bimodule, where the bimodule structure arises from the external components and the bicomodule structure comes from the tensor product. From the discussion in \cite{FR}, we know that when the maps $f,g$ satisfy certain conditions, $T_H^c(M)$ becomes a Hopf algebra with $F$ as the multiplication map. Such Hopf algebra structure on $T_H^c(M)$ is called the \textit{multi-brace cotensor Hopf algebra}.

Here we will not explain any more, but introduce one special Hopf algebra structure on $T_H^c(M)$, called the quasi-symmetric type. For any $p,q\in\bN$, denote $f_{pq},g_{pq}$ respectively as the restriction of $f,g$ on homogeneous $M^{\Box_H^p}\ull{\ot} M^{\Box_H^q}$.
\begin{defn}[{\cite[Def.7, Cor.1]{FR}}]\label{qsa}
The pair $(f,g)$ in Proposition \ref{uni} is called of the \textit{quasi-symmetric} type, if

(1) $g_{00}=m_H,~g_{pq}=0,~\forall (p,q)\neq(0,0)$, where $m_H$ is the multiplication of $H$.

(2) $f_{01}=a_L,~f_{10}=a_R,~f_{pq}=0,~\forall (p,q)\neq(0,1),(1,0),(1,1)$, where $a_L, \,a_R$ are the left and right $H$-module structure maps of $M$, respectively.

(3) $f_{11}=\al$ satisfies the following conditions:
$\al$ can factor through $M\ull{\ot}_HM$ and becomes a $H$-Hopf bimodule map.
In addition, $\al$ is associative, i.e. $\al(\al\ot\mi_M)=\al(\mi_M\ot\al)$.

When the pair $(f,g)$ is of the quasi-symmetric type, the cotensor coalgebra $T_H^c(M)$ becomes a Hopf algebra. Its Hopf subalgebra $Q_H(M)$ generated by the 0,1 homogeneous components $H,M$ is called the \textit{quantum quasi-symmetric algebra} and referred to as the QQSA for short.
\end{defn}
Moreover, the right coinvariant subspace $Q_\si(V)$ of $Q_H(M)$ is a braided Hopf algebra with $V=M^{\bs{co}\pi}$. $Q_\si(V)$ is exactly a Hopf subalgebra of the quantum quasi-shuffle algebra $T_{\bowtie_\si}(V)$ defined in \cite{JRZ}, generated by the primitive space $V$.
Conversely, $Q_H(M)$ serves as the bosonization of $Q_\si(V)$. In particular, if the map $\al=0$, then $Q_H(M)$ degenerates to be the \textit{quantum symmetric algebra} $S_H(M)$, whose right coinvariant subspace $S_\si(V)$ is just a Nichols algebra (or quantum shuffle algebra \cite{Ro1}).

\section{Miscellany for multi-parameter quantum groups}

In this section, we first recall the definition of multi-parameter quantum groups uniformly defined by Pei-Hu-Rosso in \cite{PHR}, and then give some useful results.

\noindent{3.1.}
Let $\mg_A$ be a symmetrizable Kac-Moody algebra over $\bQ$  and
$A=(a_{ij})_{i,j\in I}$ be an associated  generalized Cartan matrix.
Let $d_i$ be relatively prime positive integers such that
$d_ia_{ij}=d_{j}a_{ji}$ for $i,j\in I$. Let $\Phi$ be the root
system, $\Pi=\{\al_i\mid i\in I\}$ a set of simple roots,
$Q=\bigoplus_{i\in I}\bZ\al_i$ the root lattice, and then with
respect to $\Pi$, we have $\Phi^+$ the system of positive roots,
$Q^+=\bigoplus_{i\in I}\bZ_+\al_i$ the positive root lattice, $\La$
the weight lattice, and $\La^+$ the set of dominant weights. Let
$q_{ij}$ be indeterminates over $\bQ$ and $\bQ(q_{ij}\;|\;i,j\in I)$
be the fraction field of polynomial ring $\bQ[q_{ij}\;|\;i,j\in I]$
such that
\begin{equation}\label{car}
q_{ij}q_{ji}=q_{ii}^{a_{ij}}.
\end{equation}
Now suppose that $\bQ(q_{ij\;|\;i,j\in\ I})\subset\bK$ and
$q_{ii}^{\frac{1}{m}}\in\bK$ for $m\in\bZ_{+}$. 
Denote $\q :=(q_{ij})_{i,j\in I}$. For $n>0$, define
$$
(n)_v=\frac{v^n-1}{v-1}.
$$
$$
(n)_v!=(n)_v\cdots(2)_v(1)_v, \quad \textit{and}\quad (0)_v!=1.
$$
$$
\binom{n}{k}_v=\frac{(n)_v!}{(k)_v!(n-k)_v!}.
$$
in $\bZ[v,v^{-1}]$. When all $q_{ii}(i\in I)$ are not roots of unity, we refer to it as the \textit{generic} case.

\smallskip
\begin{defn}\label{defi} (\cite{PHR})
The multi-parameter quantum group $U_{\q}(\mg_A)$ is an associative
algebra over $\bK$ with $1$ generated by the elements $e_i, f_i,
\om_i^{\pm1}, \om_i'^{\pm1} \ (i\in I)$, subject to the relations:
\begin{eqnarray*}
 &(R1)&\quad\om_i^{\pm1}\om_j'^{\pm1}=\om_j'^{\pm1}\om_i^{\pm1},\quad
\om_i^{\pm1}\om_i^{\mp1}=\om_i'^{\pm1}\om_i'^{\mp1}=1,\\
 &(R2)&\quad\om_i^{\pm1}\om_j^{\pm1}=\om_j^{\pm1}\om_i^{\pm1},
 \quad\om_i'^{\pm1}\om_j'^{\pm1}=\om_j'^{\pm1}\om_i'^{\pm1},\\
&(R3)&\quad \om_ie_j\om_i^{-1}=q_{ij} e_j,
\qquad\om_i'e_j\om_i'^{-1}=q_{ji}^{-1} e_j,
\\
&(R4)&\quad\om_if_j\om_i^{-1}=q_{ij}^{-1} f_j,\qquad
\om_i'f_j\om_i'^{-1}=q_{ji} f_j,\\
&(R5)&\quad[\,e_i, f_j\,]=\delta_{i,j}\frac{q_{ii}}{q_{ii}-1}({\om_i-\om_i'}),\\
&(R6)&\quad \sum_{k=0}^{1-a_{ij}} (-1)^{k}
\binom{1-a_{ij}}{k}_{q_{ii}}q_{ii}^{\frac{k(k-1)}{2}}
q_{ij}^{k}e_{i}^{1-a_{ij}-k} e_{j} e_{i}^{k} =0 \quad (i\ne j),\\
&(R7)&\quad\sum_{k=0}^{1-a_{ij}} (-1)^{k}
\binom{1-a_{ij}}{k}_{q_{ii}}q_{ii}^{\frac{k(k-1)}{2}}
q_{ij}^{k}f_{i}^{k} f_{j} f_{i}^{1-a_{ij}-k}=0 \quad (i\ne j).
\end{eqnarray*}
\end{defn}

Let $U_{\q}^+$ (resp., $U_{\q}^-$) be the subalgebra of
$U_{\q}$ generated by the elements $e_i$ (resp., $f_i$) for
$i\in I$,  $U_{\q}^{+0}$ (resp., $U_{\q}^{-0}$) the
subalgebra of $U_{\q}$ generated by $\om_i^{\pm1}$ (resp.,
$\om_i'^{\pm1}$)  for $i\in I$. Let $U_{\q}^{0}$ be the subalgebra
of $U_{\q}$ generated by $\om_i^{\pm1},\om_i'^{\pm1}$  for $i\in I$.
Moreover, Let $U_{\q}^{\leq0}$ (resp., $U_{\q}^{\geq0}$) be
the subalgebra of $U_{\q}$ generated by the elements
$e_i,\om_i^{\pm1}$ for $i\in I$ (resp., $f_i, \om_i'^{\pm1}$
for $i\in I$). It is clear that $U_{\q}^0, {U}_{\q}^{\pm0}$ are
commutative algebras.  
For each $\mu\in\La$, we can
define the elements $\om_{\mu}$ and $\om_{\mu}'$ by
$$
\om_{\mu}=\prod_{i\in I}\om_{i}^{\mu_i},\qquad
\om'_{\mu}=\prod_{i\in I}{\om'_{i}}^{\mu_i}
$$
if $\mu=\sum_{i\in I}\mu_i\al_i\in \La$. For any $\mu,\ \nu\in \La$, we
denote
$$
q_{\mu\nu}:=\prod_{i,j\in I} q_{ij}^{\mu_i\nu_j}
$$
if $\mu=\sum_{i\in I}\mu_i \al_i$ and $\nu=\sum_{j\in I}\nu_j \al_j$.
Let
$$
\deg e_i=\al_i,\quad \deg f_i=-\al_i,\quad \deg \om_i^{\pm1}=\deg
\om_i'^{\pm1}=0,
$$
then one can define the $Q^{\pm}$-graded space decomposition of $U_{\q}^\pm$:
\begin{equation}\label{root}
U_{\q}^\pm=\bigoplus\limits_{\al\in Q^+}(U_{\q}^\pm)_{\pm\al}.
\end{equation}
Similarly, we have the $Q^{\pm}$-graded space decompositions of $U_{\q}^{\geq0}$ and $U_{\q}^{\leq0}$:
\begin{equation}\label{root1}
U_{\q}^{\geq0}=\bigoplus\limits_{\al\in Q^+}(U_{\q}^{\geq0})_\al,~U_{\q}^{\leq0}=\bigoplus\limits_{\al\in Q^+}(U_{\q}^{\leq0})_{-\al}.
\end{equation}

\medskip
Here list some useful results about $U_{\q}$ provided in \cite{PHR}.

\begin{prop}[{\cite[Prop.8]{PHR}}]\label{Hopf}
The associative algebra $U_{\q}(\mg)$ has a Hopf algebra structure
with the comultiplication, the counit and the antipode given by:
\begin{eqnarray*}
&& \Delta(\om_i^{\pm1})=\om_i^{\pm1}\ot\om_i^{\pm1}, \qquad
\Delta({\om_i'}^{\pm1})={\om_i'}^{\pm1}\ot{\om_i'}^{\pm1},\\
&&\Delta(e_i)=e_i\ot 1+\om_i\ot e_i, \qquad \Delta(f_i)=1\ot
f_i+f_i\ot \om_i',\\
&&\ve(\om_i^{\pm1})=\ve({\om_i'}^{\pm1})=1, \qquad\quad
\ve(e_i)=\ve(f_i)=0, \\
&&S(\om_i^{\pm1})=\om_i^{\mp1}, \qquad\qquad\qquad
S({\om_i'}^{\pm1})={\om_i'}^{\mp1},
\\
&&S(e_i)=-\om_i^{-1}e_i,\qquad\qquad\qquad
S(f_i)=-f_i\,{\om_i'}^{-1}.
\end{eqnarray*}
\end{prop}

\begin{theorem}[{\cite[Th.20]{PHR}}]\label{skew1}
There exists a unique bilinear pairing
$\lan\,,\,\ran_{\q}:\, U_{\q}^{\leq0}\times U_{\q}^{\geq0}\to \bK$
such that for all $x,\, x'\in U_{\q}^{\geq0}$, $y,\, y'\in
U_{\q}^{\leq0}$, $\mu,\, \nu\in Q$, and $i,\, j\in I$
\begin{eqnarray*}
&&\lan y,\,xx' \ran_{\q}=\lan\De(y),\,x'\ot x\ran_{\q},\\
&&\lan yy',\,x \ran_{\q}=\lan y\ot y',\,\De(x)\ran_{\q},\\
&&\lan f_i,\,e_j\ran_{\q}=\delta_{ij}\frac{q_{ii}}{1-q_{ii}},\\
&& \lan\om_{\mu}',\,\om_{\nu}\ran_{\q}=q_{\nu\mu},\\
&& \lan\om_{\mu}',\,e_i\ran_{\q}=0,\\
&& \lan f_i,\,\om_{\mu}\ran_{\q}=0.
\end{eqnarray*}
\end{theorem}

\begin{prop}[{\cite[Prop.44]{PHR}}]\label{non}
For each $\beta\in Q^+$, the restriction of
pairing $\lan\, ,\, \ran_{\q}$ to $(U^-_{\q})_{-\be}\times
(U^+_{\q})_{\be}$ is nondegenerate.
\end{prop}

\noindent{3.2.}
Moreover, we need the following important notion.
\begin{defn}
Let $C=\bigoplus_{i\in\bN} C(i)$ be a graded coalgebra with the coradical filtration $\{C_m\}_{m\in\bN}$. We say that $C$ is \textit{coradically graded} if $C_0=C(0)$, $C_1=C(0)\op C(1)$. It makes $C_m=\bigoplus\limits_{i=0}^mC(i)$.
\end{defn}


The tensor product $C\ot D$ of two graded coalgebras $C,~D$ is naturally a graded coalgebra, with $(C\ot D)(m)=\bigoplus_{i+j=m}C(i)\ot D(j)$ for any $m\in\bN$.
For a later use to fit in with the setting of multi-parameter quantum groups, we need to establish a key lemma, which is distinct from but analogous to \cite[Lemma 2.3]{CM}.
\begin{lem}\label{cog}
Let $C,D$ be two pointed coalgebras over $\bK$. If they are coradically graded, then the tensor product $C\ot D$ is also coradically graded.
\end{lem}
\begin{proof}
First using \cite[Lemma 5.1.10]{Mon}, we see that $(C\ot D)_0=C_0\ot D_0=C(0)\ot D(0)=(C\ot D)(0)$. Now we need to prove that $(C\ot D)_1=(C\ot D)(0)\op(C\ot D)(1)$.  By the Taft-Wilson theorem~\cite[Theorem 5.4.1]{Mon}, it reduces to showing that any skew-primitive element $b\in C\ot D$ lies in $(C\ot D)(0)\op(C\ot D)(1)$. Suppose there exist $g,g'\in G(C),h,h'\in G(D)$ such that $\De_{C\ot D}(b)=(g\ot h)\ull{\ot}b+b\ull{\ot}(g'\ot h')$. Since the coproduct $\De_{C\ot D}$ is compatible with the grading, we can also assume that $b\in C(m)\ot D(n)$ for some $m,n\in\bN$, and write
\[b=\sum_{i,j}c_i\ot d_j,\]
where $\{c_i\}\subset C(m),\{d_j\}\subset D(n)$ both are linearly independent and
\[\De_C(c_i)=g\ot c_i+c_i\ot g'+x_i, \quad \De_D(d_j)=h\ot d_j+d_j\ot h'+y_j\]
for some $x_i\in\sum\limits_{k=1}^{m-1} C(k)\ot C(m-k)$ and $y_j\in\sum\limits_{k=1}^{n-1} D(k)\ot D(n-k)$. Now expanding $\De_{C\ot D}(b)$, we find a nonzero term
\[\sum_{i,j}(c_i\ot h)\ull{\ot}(g'\ot d_j)\]
in the homogenous component $(C(m)\ot D(0))\ull{\ot}(C(0)\ot D(n))$. Since $b$ is $(g\ot h,g'\ot h')$-primitive and $C,D$ both are coradically graded, it forces $m=0,n\leq1$ or $m\leq1,n=0$.
\end{proof}

\begin{prop}[{\cite{Ni} Milnor-Moore type property}]\label{MM} Let $B=\bigoplus_{i\in\bN}B(i)$ be a graded Hopf algebra. $B(1)$ has the natural $B(0)$-Hopf bimodule structure. Now denote $\mb{Prim}(B)=\De^{-1}(B(0)\ot B+ B\ot B(0))$, then $B$ is isomorphic to the quantum symmetric algebra $S_{B(0)}(B(1))$ if and only if
\[\begin{array}{l}
\text{(\mb i)}\;\quad\mb{Prim}(B)=B(0)\op B(1),\\
\text{(\mb ii)}\quad(\bigoplus_{i\geq1}B(i))^2=\bigoplus_{i\geq2}B(i).
\end{array}\]
Moreover, if the coradical $B_0=B(0)$, then $B$ is coradically graded as $\mb{Prim}(B)=B_1$.
\end{prop}

\noindent{3.3.}
Now let $\{(U_{\q})_m\}_{m\in\bN}$ be the coradical filtration of $U_{\q}$.
Here we figure out $(U_{\q})_1$, which actually generalizes the result about the one-parameter version $U_q(\mg)$, $\mg$ semisimple,
given by Chin and Musson~\cite[Theorem A]{CM}.
\begin{prop}\label{cor}
For the generic case, $(U_{\q})_1=U_{\q}^0+\sum_{i\in I}e_iU_{\q}^0+f_iU_{\q}^0$.
\end{prop}
\begin{proof}
First, we deal with $U_\q^{\geq0}$ and $U_\q^{\leq0}$. For any $\al\in Q^+$, one can define the \textit{height}
of $\al$ as $\mb{ht}(\al)=\sum_{i\in I}m_i$ if $\al=\sum_{i\in I}m_i\al_i$.
Note that the Borel positive part $U_\q^{\geq0}$ becomes a graded Hopf algebra with respect to the height, i.e.,
\[U_\q^{\geq0}=\bigoplus_{k\in\bN}U_\q^{\geq0}(k),\]
where $U_\q^{\geq0}(k)=\bigoplus\limits_{\al\in Q^+,\bs{ht}(\al)=k}(U_\q^{\geq0})_\al,~\forall k\in\bN$.
Next, we show that $U_\q^{\geq0}$ is isomorphic to the quantum symmetric algebra $S_{H^+}(M^+)$,
where $H^+=(U_\q^{\geq0})_0=U_\q^{\geq0}(0)=U_\q^{+0},~M^+=U_\q^{\geq0}(1)=\bigoplus_{i\in I}(U_\q^{\geq0})_{\al_i}=\bigoplus_{i\in I}e_iU_\q^{+0}$.
It suffices to check that $U_\q^{\geq0}$ satisfies the conditions (i) \& (ii) in Proposition \ref{MM}.
Apparently, the ideal $\bigoplus_{k\geq1}(U_\q^{\geq0})(k)$ is generated by $(U_\q^{\geq0})(1)$,
thus (ii) is satisfied. For condition (i), we consider the right coinvariant subspace of $U_\q^{\geq0}$ with respect to $U_\q^{+0}$, that is $U_\q^+$.
Note that any skew primitive element in $U_\q^+$ with height not smaller than $2$ lies in the radical of the skew pairing $\lan\,,\,\ran_{\q}$.
Due to Proposition \ref{non}, we know that it is zero. Now $U_\q^{\geq0}$ is the bosonization of $U_\q^+$, denoted by $U_\q^{+0}$,
thus condition (i) for $U_\q^{\geq0}$ also holds. Similarly,
\[U_\q^{\leq0}=\bigoplus_{k\in\bN}U_\q^{\leq0}(k),\]
where $U_\q^{\geq0}(k)=\bigoplus\limits_{\al\in Q^-,\bs{ht}(-\al)=k}(U_\q^{\leq0})_\al,~\forall k\in\bN$. $U_\q^{\leq0}$
is isomorphic to the quantum symmetric algebra $S_{H^-}(M^-)$, where $H^-=(U_\q^{\leq0})_0=U_\q^{\leq0}(0)=U_\q^{-0},~M^-=U_\q^{\leq0}(1)=\bigoplus_{i\in I}(U_\q^{\leq0})_{-\al_i}=\bigoplus_{i\in I}f_iU_\q^{-0}$.
Moreover, it is clear that $U_\q^{\leq0},U_\q^{\geq0}$ both are coradically graded.

Now we return to describe $(U_\q)_1$. Note that $U_{\q}$ is naturally isomorphic to $U^{\leq0}_{\q}\ot U^{\geq0}_{\q}$ as graded coalgebras, since $U_{\q}$ can be constructed as the Drinfeld double $\mathcal{D}(U_{\q}^{\geq0},U_{\q}^{\leq0},\lan,\ran_{\q})$~\cite[Coro. 21]{PHR}. More explicitly,
\begin{equation}\label{gra}
U_\q=\bigoplus_{k\in\bN}U_\q(k),
\end{equation}
where $U_\q(k)=\bigoplus(U_\q^{\leq0})_\al(U_\q^{\geq0})_\be$ with the sum over all $\al\in Q^-,\be\in Q^+,~\mb{ht}(-\al)+\mb{ht}(\be)=k$ for any $k\in\bN$. Hence, $U_\q$ is coradically graded due to Lemma \ref{cog}. Meanwhile, it is clear that $(U_{\q})_0=U_{\q}^0=U_\q(0)$ and $U_\q(1)=\sum_{i\in I}e_iU_{\q}^0+f_iU_{\q}^0$, thus we see that $(U_{\q})_1=U_{\q}^0+\sum_{i\in I}e_iU_{\q}^0+f_iU_{\q}^0$.
\end{proof}

By Proposition \ref{cor}, we can also obtain the following useful statement, analogous to that given in \cite[Theorem C]{CM}.
\begin{prop}
For the generic case, if $L$ is a Hopf ideal of $U_{\q}$ such that $L\cap U_{\q}^0=0$, then $L=0$.
\end{prop}
\begin{proof}
Let
$\pi$ be the projection map from $U_{\q}$ to $\ol{U_{\q}}$, where $\ol{U_{\q}}=U_{\q}/L$.
First note that the restriction $\pi|_{U_{\q}^0}$ is injective by the assumption of $L$. If we assume that $L\neq0$, then $\pi|_{(U_{\q})_1}$ can not be injective, i.e., $L\cap(U_{\q})_1\neq0$, by a result of Heyneman-Radford on injectivity of a coalgebra map (\cite[Theorem 5.3.1]{Mon}). On the other hand, the coideal $L$ is homogeneous with respect to the grading of $U_\q$ in (\ref{gra}), thus the images $\ol{e_i}\ol{U^0_{\q}}\cap\ol{U^0_{\q}}=0~(i\in I)$ in $\ol{U_{\q}}$. Now $(U_{\q})_1=U_{\q}^0+\sum_{i\in I}e_iU_{\q}^0+f_iU_{\q}^0$ by Proposition \ref{cor}, thus we can assume $L\cap e_iU_{\q}^0\neq0$ for some $i\in I$ without loss of generality.

Take $\sum_j a_j e_ig_j\in L$, where the $g_j\in U_{\q}^0$ are distinct group-like elements and the coefficients $a_j$ are nonzero, then
\[0=\sum_j a_j\De(\ol{e_ig_j})=\sum_j a_j \lb\ol{e_ig_j}\ot\ol{g_j}+\ol{\om_ig_j}\ot\ol{e_ig_j}\rb.\]
Since $\ol{e_i}\ol{U^0_{\q}}\cap\ol{U^0_{\q}}=0$ and $\{\ol{g_j}\}$ is linearly independent in $\ol{U^0_{\q}}$, we conclude that $\ol{e_ig_j}=0$, i.e. $e_ig_j\in L$ for all $j$. In particular, $e_i\in L$, thus $\om_i-\om_i'=\frac{q_{ii}-1}{q_{ii}}[\,e_i, f_i\,]\in L$, which contradicts to the assumption $L\cap U_{\q}^0=0$.
\end{proof}

\section{QQSA realization of multi-parameter quantum groups}

\noindent{4.1.}
Inspired by the work of Fang-Rosso \cite{FR}, we will describe axiomatically the entire multi-parameter quantum groups of symmetrizable Kac-Moody algebras $U_{\q}$ as quantum
quasi-symmetric algebras as follows.

Let $H=\bK[K_i^{\pm1},{K'_i}^{\pm1}:i\in I]$ be the group algebra of the free abelian group $\bZ^{2|I|}$ with the canonical Hopf algebra structure.
Take $W$ to be a $\bK$-vector space spanned by a basis $\{E_i,F_i,\xi_i:i\in I\}$,  $M=W\ot H$.
Consider the $H$-Hopf bimodule structure on $M$ as follows.

(1) $M$ has the trivial right $H$-Hopf module structure, i.e.
\[(x\ot K)\cdot K'=x\ot KK',~\de_R(x\ot K)=(x\ot K_{(1)})\ot K_{(2)},\quad x\in W,~K,K'\in H.\]

(2) The left $H$-Hopf module structure of $M$ is defined as follows:
First endow $W$ with the following left $H$-module and comodule structure,
\[\begin{array}{l}
K_i\cdot E_j=q_{ij}E_j,~K_i\cdot F_j=q_{ij}^{-1}F_j,~K'_i\cdot F_j=q_{ji}F_j,~K'_i\cdot E_j=q_{ji}^{-1}E_j,\\
K_i\cdot\xi_j=\xi_j,~K'_i\cdot\xi_j=\xi_j;\\
\de_L(E_i)=K_i\ot E_i,~\de_L(F_i)={K'_i}^{-1}\ot F_i,~\de_L(\xi_i)=K_i{K'_i}^{-1}\ot\xi_i.
\end{array}
\]
Meanwhile, $H$ derives the left $H$-module and comodule structure from the left multiplication and its comultiplication,
then so does $M=W\ot H$ as the tensor product. Moreover, one can easily check that it makes $M$ a left $H$-Hopf module.

Now we write $x\ot K\in M$ as $xK$ for short. Define the map
\begin{equation}\label{map}
\al=\al_M:M\ot M\rw M
\end{equation}
as follows: For any $K,K'\in H$, if $\la\in \bK$ is the constant such that $K\cdot F_j=\la F_j$, then
\[\al(E_iK\ot F_jK')=\de_{ij}\dfrac{q_{ii}\la\xi_iKK'}{q_{ii}-1}, \quad i, j\in I,\]
and for any other elements not of the above form, $\al$ gives 0. Here we only need to check that $\al$ satisfies the condition (3) given in
Definition \ref{qsa} on those elements $E_iK\ot F_jK'$:

(1) $\al$ can factor through $M\ot_H M$. Without loss of generality, we assume that $K$ is group-like, then
\[\al(E_i\ot(K\cdot F_j)KK')=\la\al(E_i\ot F_jKK')=\de_{ij}\dfrac{q_{ii}\la \xi_iKK'}{q_{ii}-1}=\al(E_iK\ot F_jK').\]

(2) $\al$ is a $H$-bimodule map. First, it is obvious a right $H$-module map. For the left module case, take any $K_p,K'_p,~p\in I$, then
\[\begin{array}{l}
\begin{split}
\al(K_p\cdot(E_iK\ot F_jK'))&=\al((K_p\cdot E_i)K_pK\ot F_jK')=\de_{ij}\dfrac{q_{pi}q_{ii}q_{pj}^{-1}\la\xi_iK_pKK'}{q_{ii}-1}\\
&=\de_{ij}\dfrac{q_{ii}\la\xi_iK_pKK'}{q_{ii}-1}=K_p\cdot\al(E_iK\ot F_jK'),
\end{split}\\
\begin{split}
\al(K'_p\cdot(E_iK\ot F_jK'))&=\al((K'_p\cdot E_i)K'_pK\ot F_jK')=\de_{ij}\dfrac{q_{ip}^{-1}q_{ii}q_{jp}\la\xi_iK'_pKK'}{q_{ii}-1}\\
&=\de_{ij}\dfrac{q_{ii}\la\xi_iK'_pKK'}{q_{ii}-1}=K'_p\cdot\al(E_iK\ot F_jK').
\end{split}
\end{array}\]

(3) $\al$ is a $H$-cobimodule map. We can still assume that $K$ is group-like, and $K\cdot F_j=\la F_j$. For the left comodule structure,
\[\begin{split}
(\mi\ot\al)&\de_L(E_iK\ot F_jK')=K_i{K'_j}^{-1}KK'_{(1)}\ot\dfrac{\de_{ij}q_{ii}\la\xi_iKK'_{(2)}}{q_{ii}-1}\\
&=K_i{K'_i}^{-1}KK'_{(1)}\ot\dfrac{\de_{ij}q_{ii}\la\xi_iKK'_{(2)}}{q_{ii}-1}=\de_L\al(E_iK\ot F_jK').
\end{split}\]
For the right comodule structure,
\[(\al\ot\mi)\de_R(E_iK\ot F_jK')=\dfrac{\de_{ij}q_{ii}\la\xi_iKK'_{(2)}}{q_{ii}-1}\ot KK'_{(2)}=\de_R\al(E_iK\ot F_jK').\]

(4) The associativity of $\al$ can be seen immediately from definition.

According to Definition \ref{qsa}, if take $f_{11}=\al,~f_{01}=a_L,~f_{10}=a_R$, then it defines a Hopf algebra $Q_H(M)$,
where the multiplication is denoted by $*$. Consider the ideal $J$ of $Q_H(M)$ generated by $\{\xi_i-K_i{K'_i}^{-1}+1:i\in I\}$,
then $J$ is a Hopf ideal. Indeed,
\[\De(\xi_i-K_i{K'_i}^{-1}+1)=(\xi_i-K_i{K'_i}^{-1}+1)\ot1+K_i{K'_i}^{-1}\ot(\xi_i-K_i{K'_i}^{-1}+1).\]
Define the Hopf quotient $U_H(M)=Q_H(M)/J$, and we will prove that it is isomorphic to the multi-parameter quantum group $U_{\q}(\mg_A)$ as Hopf algebras.
\begin{theorem}\label{iso}
For the generic case, there exists a Hopf algebra isomorphism $\psi:U_{\q}(\mg_A)\rw U_H(M)$ such that for any $i\in I$,
\[\psi(\om_i)=K_i,~\psi(\om'_i)=K'_i,~\psi(e_i)=E_i,~\psi(f_i)=F_i*K'_i=F_iK'_i.\]
\end{theorem}
\begin{proof}
Take a set $X=\{e_i, f_i,
\om_i^{\pm1}, \om_i'^{\pm1}:i\in I\}$. By the universal property of the tensor algebra $T(X)$ on $X$, one can define an algebra map $\Psi:T(X)\rw U_H(M)$,
\[\Psi(\om_i)=K_i,~\Psi(\om'_i)=K'_i,~\Psi(e_i)=E_i,~\Psi(f_i)=F_i*K'_i=F_iK'_i.\]
Now we endow $T(X)$ with the coproduct $\De$ and the antipode $S$ as in Proposition \ref{Hopf}. Since
\[\begin{array}{l}
\De(K_i)=K_i\ull{\ot}K_i,~\De(K_i)=K'_i\ull{\ot}K'_i,~\De(E_i)=E_i\ull{\ot}1+K_i\ull{\ot}E_i,\\
\De(F_i*K'_i)=(F_i\ull{\ot}1+{K'_i}^{-1}\ull{\ot}F_i)*(K'_i\ull{\ot}K'_i)=
F_iK'_i\ull{\ot}K'_i+1\ull{\ot}F_iK'_i,
\end{array}\]
we know that $\Psi$ becomes a Hopf algebra homomorphism.

Let $R$ be the ideal of $T(X)$ defined by the relations (R1) -- (R7), then $U_{\q}(\mg_A)=T(X)/R$.
Now we need to show that $\Psi$ maps the ideal $R$ to $0$, thus it induces the Hopf algebra homomorphism $\psi$.
For any $K, K', K''\in H, ~x, y\in W$, we list some useful equalities due to the formula (\ref{mu}) as follows.
\[\begin{array}{l}
K*K'=KK',~K*x=(K\cdot x)K,~x*K=xK,\\
\begin{split}
xK*yK'&=(xK_{(1)}\cdot y_{(-1)}K'_{(1)})\ot(K_{(2)}\cdot y_{(0)}K'_{(2)})\\
&+(x_{(-1)}K_{(1)}\cdot yK'_{(1)})\ot(x_{(0)}K_{(2)}\cdot K'_{(2)})+\al(xK\ot yK'),
\end{split}\\
K*(xK'\ot yK'')=(K_{(1)}\cdot xK')\ot(K_{(2)}\cdot yK''),\\
(xK\ot yK')*K''=xKK''_{(1)}\ot yK'k''_{(2)}.
\end{array}
\]
First it is clear that $\Psi$ annihilates (R1), (R2). For (R3), we have
\[\begin{array}{l}
\Psi(\om_ie_j-q_{ij}e_j\om_i)=K_i*E_j-q_{ij}E_j*K_i=(K_i\cdot E_j)K_i-q_{ij}E_jK_i=0,\\
\Psi(\om'_ie_j-q_{ji}^{-1}e_j\om'_i)=K'_i*E_j-q_{ji}^{-1}E_j*K'_i=(K'_i\cdot E_j)K'_i-q_{ji}^{-1}E_jK'_i=0.
\end{array}\]
For (R4),
\[\begin{array}{l}
\begin{split}
\Psi(\om_if_j-q_{ij}^{-1}f_j\om_i)&=K_i*F_j*K'_j-q_{ij}^{-1}F_j*K'_j*K_i\\
&=(K_i\cdot F_j)K_iK'_j-q_{ij}^{-1}F_jK'_jK_i=0,
\end{split}\\
\begin{split}
\Psi(\om'_if_j-q_{ji}f_j\om'_i)&=K'_i*F_j*K'_j-q_{ji}F_j*K'_j*K'_i\\
&=(K'_i\cdot F_j)K'_iK'_j-q_{ji}F_jK'_jK'_i=0.
\end{split}
\end{array}\]
For (R5),
\[\begin{split}
\Psi([e_i,f_j]&-\de_{ij}\frac{q_{ii}}{q_{ii}-1}(\om_i-\om_i'))=
E_i*F_j*K'_j-F_j*K'_j*E_i-\de_{ij}\frac{q_{ii}}{q_{ii}-1}(K_i-K_i')\\
&=(E_i\cdot{K'_j}^{-1}\ot F_j+K_i\cdot F_j\ot E_i+\al(E_i\ot F_j))*K'_j\\
&-(F_jK'_j\cdot K_i\ot K'_j\cdot E_i+{K'_j}^{-1}K'_j\cdot E_i\ot F_jK'_j+\al(F_jK'_j\ot E_i))\\
&-\de_{ij}\frac{q_{ii}}{q_{ii}-1}(K_i-K_i')\\
&=E_i\ot F_jK'_j+q_{ij}^{-1}F_jK_iK'_j\ot E_iK'_j+\de_{ij}\dfrac{q_{ii}\xi_iK'_j}{q_{ii}-1}\\
&-(q_{ij}^{-1}F_jK'_jK_i\ot E_iK'_j+E_i\ot F_jK'_j)-\de_{ij}\frac{q_{ii}}{q_{ii}-1}(K_i-K_i')=0.
\end{split}
\]
What's left is to check the quantum Serre relations (R6), (R7). If abbreviate the LHS of (R6) as $u^+_{ij}$, then $u^+_{ij}=\mb{ad}_l(e_i)^{1-a_{ij}}(e_j)$, where $\mb{ad}_l(x)(y)=x_{(1)}yS(x_{(2)})$. As $\Psi$ is a Hopf algebra homomorphism,
\[\begin{split}
\Psi(\mb{ad}_l(e_i)(e_j))&=\mb{ad}_l(E_i)(E_j)=E_i*E_j-K_i*E_j*K_i^{-1}*E_i\\
&=E_i\cdot K_j\ot E_j+K_i\cdot E_j\ot E_i-q_{ij}(E_j\cdot K_i\ot E_i+K_j\cdot E_i\ot E_j)\\
&=(1-q_{ij}q_{ji})E_iK_j\ot E_j.
\end{split}\]
Now consider $\mb{ad}_l(E_i)^s(E_j),~s\geq1$ in general. By the formula (\ref{mu}),
\[\begin{split}
E_i&*(E_iK_i^{s-1}K_j\ot\cdots\ot E_iK_j\ot E_j)=K_i\cdot E_iK_i^{s-1}K_j\ot\cdots\ot K_i\cdot E_iK_j\ot K_i\cdot E_j\ot E_i\\
&+\sum_{k=0}^sK_i\cdot E_iK_i^{s-1}K_j\ot\cdots\ot K_i\cdot E_iK_i^kK_j\ot E_i\cdot K_i^kK_j\ot E_iK_i^{k-1}K_j\ot\cdots \ot E_iK_j\ot E_j\\
&=q_{ii}^sq_{ij}E_iK_i^sK_j\ot\cdots\ot E_iK_iK_j\ot E_jK_i\ot E_i+(s+1)_{q_{ii}}E_iK_i^sK_j\ot\cdots\ot E_iK_j\ot E_j,
\end{split}\]
\[\begin{split}
K_i&*(E_iK_i^{s-1}K_j\ot\cdots\ot E_iK_j\ot E_j)*K_i^{-1}*E_i\\
&=q_{ii}^sq_{ij}(E_iK_i^{s-1}K_j\ot\cdots\ot E_iK_j\ot E_j)*E_i\\
&=q_{ii}^sq_{ij}\big(E_iK_i^{s-1}K_j\cdot K_i\ot\cdots\ot E_iK_j\cdot K_i\ot E_j\cdot K_i\ot E_i\\
&+\sum_{k=0}^sE_iK_i^{s-1}K_j\cdot K_i\ot\cdots\ot E_iK_i^kK_j\cdot K_i\ot K_i^kK_j\cdot E_i\ot E_iK_i^{k-1}K_j\ot\cdots\ot E_iK_j\ot E_j\big)\\
&=q_{ii}^sq_{ij}\lb E_iK_i^sK_j\ot\cdots\ot E_iK_iK_j\ot E_jK_i\ot E_i+(s+1)_{q_{ii}}q_{ji}E_iK_i^sK_j\ot\cdots\ot E_iK_j\ot E_j\rb.
\end{split}\]
Subtracting one of two equalities above from the other, we get $\mb{ad}_l(E_i)(E_iK_i^{s-1}K_j\ot\cdots\ot E_iK_j\ot E_j)=(s+1)_{q_{ii}}(1-q_{ii}^sq_{ij}q_{ji})E_iK_i^sK_j\ot\cdots\ot E_iK_j\ot E_j$. Hence,
\[\mb{ad}_l(E_i)^s(E_j)=(s)_{q_{ii}}!\prod_{k=1}^{s-1}(1-q_{ii}^kq_{ij}q_{ji})
E_iK_i^{s-1}K_j\ot\cdots\ot E_iK_j\ot E_j,~s\geq1.\]
Especially for $s=1-a_{ij}$, the constraint (\ref{car}) makes $\Psi(u^+_{ij})=\mb{ad}_l(E_i)^{1-a_{ij}}(E_j)=0$.

Similarly, we abbreviate the LHS of (R7) as ~$u^-_{ij}$, then $u^-_{ij}=\mb{ad}_r(f_i)^{1-a_{ij}}(f_j)$, where $\mb{ad}_r(x)(y)=S(x_{(1)})yx_{(2)}$. As $\Psi$ is a Hopf algebra homomorphism,
\[\begin{split}
\Psi(\mb{ad}_r(f_i)(f_j))&=\mb{ad}_r(F_iK'_i)(F_jK'_j)=F_jK'_j*F_iK'_i
-F_i*F_jK'_j*K'_i\\
&=F_jK'_j\cdot {K'_i}^{-1}K'_i\ot K'_j\cdot F_iK'_i+{K'_j}^{-1}K'_j\cdot F_iK'_i\ot F_jK'_j\cdot K'_i\\
&-(F_i\cdot {K'_j}^{-1}K'_jK'_i\ot F_jK'_jK'_i+{K'_i}^{-1}\cdot F_jK'_jK'_i\ot F_i\cdot K'_jK'_i)\\
&=(q_{ij}-q_{ji}^{-1})F_jK'_j\ot F_iK'_iK'_j.
\end{split}\]
In general for $\mb{ad}_r(F_iK'_i)^s(F_jK'_j),~s\geq1$, by the formula (\ref{mu}) again, we have
\[\begin{split}
(F_j&K'_j\ot F_iK'_iK'_j\cdots\ot F_i{K'_i}^sK'_j)*F_iK'_i\\
&=F_iK'_i\ot F_jK'_j\cdot K'_i\ot F_iK'_iK'_j\cdot K'_i\ot\cdots\ot F_i{K'_i}^sK'_j\cdot K'_i\\
&+\sum_{k=0}^sF_jK'_j\ot F_iK'_iK'_j\ot\cdots\ot F_i{K'_i}^kK'_j\ot {K'_i}^kK'_j\cdot F_iK'_i\\
&\ot F_i{K'_i}^{k+1}K'_j\cdot K'_i\ot\cdots\ot F_i{K'_i}^sK'_j\cdot K'_i\\
&=F_iK'_i\ot F_jK'_iK'_j\ot F_i{K'_i}^2K'_j\ot\cdots\ot F_i{K'_i}^{s+1}K'_j\\
&+(s+1)_{q_{ii}}q_{ij}F_jK'_j\ot F_iK'_iK'_j\cdots\ot F_i{K'_i}^{s+1}K'_j,
\end{split}\]
\[\begin{split}
F_i&*(F_jK'_j\ot F_iK'_iK'_j\cdots\ot F_i{K'_i}^sK'_j)*K'_i\\
&=F_i*(F_jK'_jK'_i\ot F_i{K'_i}^2K'_j\cdots\ot F_i{K'_i}^{s+1}K'_j)\\
&=F_i\cdot K'_i\ot F_jK'_jK'_i\ot F_i{K'_i}^2K'_j\ot\cdots\ot F_i{K'_i}^{s+1}K'_j\\
&+\sum_{k=0}^s{K'_i}^{-1}\cdot F_jK'_jK'_i\ot{K'_i}^{-1}\cdot F_i{K'_i}^2K'_j\ot\cdots\ot{K'_i}^{-1}\cdot F_i{K'_i}^{k+1}K'_j\\
&\ot F_i\cdot {K'_i}^{k+1}K'_j\ot F_i{K'_i}^{k+2}K'_j\ot\cdots\ot F_i{K'_i}^{s+1}K'_j\\
&=F_iK'_i\ot F_jK'_iK'_j\ot F_i{K'_i}^2K'_j\ot\cdots\ot F_i{K'_i}^{s+1}K'_j\\
&+(s+1)_{q_{ii}}q_{ii}^{-s}q_{ji}^{-1}F_jK'_j\ot F_iK'_iK'_j\cdots\ot F_i{K'_i}^{s+1}K'_j.
\end{split}\]
Subtracting one of two equalities above from the other, it gives $\mb{ad}_r(F_iK'_i)(F_jK'_j\ot F_iK'_iK'_j\cdots\ot F_i{K'_i}^sK'_j)=(s+1)_{q_{ii}}(q_{ij}-q_{ii}^{-s}q_{ji}^{-1})F_jK'_j\ot F_iK'_iK'_j\cdots\ot F_i{K'_i}^{s+1}K'_j$. Hence,
\[\mb{ad}_r(F_iK'_i)^s(F_jK'_j)=(s)_{q_{ii}}!\prod_{k=1}^{s-1}(q_{ij}-q_{ii}^{-k}q_{ji}^{-1})
F_jK'_j\ot F_iK'_iK'_j\cdots\ot F_i{K'_i}^sK'_j,~s\geq1.\]
Especially for $s=1-a_{ij}$, the constraint (\ref{car}) gives $\Psi(u^-_{ij})=\mb{ad}_r(F_iK'_i)^{1-a_{ij}}(F_jK'_j)=0$.

So far we have proved the existence of the Hopf algebra homomorphism $\psi:U_{\q}(\mg_A)\rw U_H(M)$.
Since $U_H(M)$ has the generator set $\{E_i,F_i,{K_i}^{\pm1},{K'_i}^{\pm1}:i\in I\}$,
in which all the elements lie in the image of $\psi$, $\psi$ is surjective.
In order to prove that $\psi$ is also injective,
we note that $(U_{\q})_1=U_{\q}^0+\sum_{i\in I}e_iU_{\q}^0+f_iU_{\q}^0$ by Proposition \ref{cor}.
From the definition of $\psi$, $\psi((U_{\q})_1)=(H+\sum_{i\in I}E_iH+F_iH)+J$ and
it is clear that the restriction $\psi|_{(U_{\q})_1}$ is injective, thus
$\psi$ is also injective due to a theorem of Heyneman and Radford~\cite[Theorem 5.3.1]{Mon}.
\end{proof}

\begin{rem}
When $q_{ii},~i\in I$ are all roots of unity, $U_H(M)$ is actually isomorphic to the small quantum group $u_{\q}(\mg_A)$. Since for any $r\geq1$, we have
\[\begin{split}
E_i&*(E_i{K_i}^{r-1}\ot\cdots\ot E_iK_i\ot E_i)=\sum_{k=0}^rK_i\cdot E_i{K_i}^{r-1}\ot\cdots\ot K_i\cdot E_i{K_i}^{r-k}\\
&\ot E_i\cdot K_i^{r-k}\ot E_iK_i^{r-1-k}\ot\cdots\ot E_i=(r+1)_{q_{ii}}E_i{K_i}^r\ot\cdots\ot E_iK_i\ot E_i,
\end{split}\]
i.e. $E_i^{*r}=(r)_{q_{ii}}!E_i{K_i}^{r-1}\ot\cdots\ot E_iK_i\ot E_i,~r\geq1$. Especially for $r=\mb{ord}(q_{ii})$, we get $E_i^{*r}=0$. Similarly,
\[\begin{split}
F_iK'_i&*(F_iK'_i\ot F_i{K'_i}^2\ot\cdots\ot F_i{K'_i}^r)=\sum_{k=0}^rF_iK'_i\ot\cdots\ot F_i{K'_i}^k\ot F_iK'_i\cdot {K'_i}^k\\
&\ot K'_i\cdot F_i{K'_i}^{k+1}\ot\cdots\ot K'_i\cdot F_i{K'_i}^r=(r+1)_{q_{ii}}F_iK'_i\ot F_i{K'_i}^2\ot\cdots\ot F_i{K'_i}^{r+1},
\end{split}\]
i.e. $(F_iK'_i)^{*r}=(r)_{q_{ii}}!F_iK'_i\ot F_i{K'_i}^2\ot\cdots\ot F_i{K'_i}^r,~r\geq1$. Especially for $r=\mb{ord}(q_{ii})$, $(F_iK'_i)^{*r}=0$. On the other hand, $K_i^{*r}=K_i^r,~{K'_i}^{*r}={K'_i}^r$, hence $H$ should be redefined as
\[\bK[K_i^{\pm1},{K'_i}^{\pm1}:i\in I]/(K_i^{r_i}-1,{K'_i}^{r_i}-1:i\in I),~r_i=\mb{ord}(q_{ii}), i\in I,\]
the group algebra of $\prod_{i\in I}(\bZ/r_i)^2$.
\end{rem}

\noindent{4.2.}
Now suppose there exists another matrix $\lb\hat{q}_{ij}\rb_{i,j\in I}$ over $\bK$ satisfying
\begin{equation}\label{con}
\hat{q}_{ii}=q_{ii},~\hat{q}_{ij}\hat{q}_{ji}=q_{ij}q_{ji},~\forall i,j\in I.
\end{equation}
Define a Yetter-Drinfel'd $H$-module $\hat{W}=\mb{span}_\bK\{\hat{E}_i,\hat{F}_i,\hat{\xi}_i:i\in I\}$ analogous to $W$,
by replacing all the structure constants $q_{ij},i,j\in I$ with $\hat{q}_{ij},i,j\in I$.
That gives the $H$-Hopf bimodule $\hat{M}=\hat{W}\ot H$. Inspired by Proposition 3.9 in \cite{AuS}, we fix a group bicharacter $\si:\Ga\times\Ga\rw\bK^\times$ on $\Ga=\bZ^{2|I|}$ satisfying
\[\begin{array}{l}
\si(K_i,K_j)\si^{-1}(K_j,K_i)=\hat{q}_{ij}q_{ij}^{-1},\\
\si(K'_i,K'_j)\si^{-1}(K'_j,K'_i)=\hat{q}_{ij}q_{ij}^{-1},\\
\si(K'_i,K_j)\si^{-1}(K_j,K'_i)=\hat{q}_{ji}^{-1}q_{ji},
\end{array}\]
for all $i,j\in I$. Since a group bicharacter is naturally a group cocycle, we can linearly extend $\si$ to be a Hopf cocycle on $H$.
Recall that given $V\in{_H^H\mathcal {Y}\mathcal {D}}$, the coaction of $H^\si$ on $V^\si$ is unchanged,
but the action of $H^\si$ on $V^\si$ varies as follows~\cite[Th.2.7]{mo},
\[h\cdot_\si v=\si(h_{(1)},v_{(-1)})\si^{-1}((h_{(2)}\cdot v_{(0)})_{(-1)},h_{(3)})(h_{(2)}\cdot v_{(0)})_{(0)}, ~h\in H, v\in V.\]
In particular, for $W^\si$ we have the following formulas,
\[\begin{array}{l}
K_i\cdot_\si E_j=\si(K_i,K_j)\si^{-1}(K_j,K_i)K_i\cdot E_j=\hat{q}_{ij}E_j,\\
K'_i\cdot_\si E_j=\si(K'_i,K_j)\si^{-1}(K_j,K'_i)K'_i\cdot E_j=\hat{q}^{-1}_{ji}E_j,\\
K_i\cdot_\si F_j=\si(K_i,{K'_j}^{-1})\si^{-1}({K'_j}^{-1},K_i)K_i\cdot F_j=\hat{q}_{ij}^{-1}F_j,\\
K'_i\cdot_\si F_j=\si(K'_i,{K'_j}^{-1})\si^{-1}({K'_j}^{-1},K'_i)K'_i\cdot F_j=\hat{q}_{ji}F_j,\\
K_i\cdot_\si\xi_j=\si(K_i,K_j{K'_j}^{-1})\si^{-1}(K_j{K'_j}^{-1},K_i)K_i\cdot
\xi_j=\xi_j,\\
K'_i\cdot_\si\xi_j=\si(K'_i,K_j{K'_j}^{-1})\si^{-1}(K_j{K'_j}^{-1},K'_i)K'_i\cdot
\xi_j=\xi_j,\\
\end{array}\]
for all $i, j\in I$. As the cocycle twist $H^\si=H$, we thus obtain the following result.
\begin{prop}\label{yd}
There exists an isomorphism between $\hat{W}$ and $W^\si$ in $_H^H\mathcal {Y}\mathcal {D}$,
mapping $\hat{E}_i,\hat{F}_i,\hat{\xi}_i$ to $E_i,F_i,\xi_i$ respectively for all $i\in I$.
\end{prop}

\begin{rem}
One can also twist the $H$-Hopf bimodule $M$ via $\si$ directly,
and the two-sided actions of $H^\si$ on $M^\si$ are as follows~\cite[Th.2.5]{mo},
\[\begin{array}{l}
h\bullet_\si m=\si(h_{(1)},m_{(-1)})\si^{-1}(h_{(3)},m_{(1)})h_{(2)}\bullet m_{(0)},\\
m\bullet_\si h=\si(m_{(-1)},h_{(1)})\si^{-1}(m_{(1)},h_{(3)})m_{(0)}\bullet h_{(2)},
\end{array}\]
for all $h\in H,m\in M$. Now we have the following isomorphism of $H$-Hopf bimodules.
\[\phi:M^\si\rw W^{\si}\ot H,~x\ot K\mapsto \si(x_{(-1)},K^{-1})x_{(0)}\ot K, ~x\in W, K\in\Ga.\]
We also abuse the notation $\phi$ to denote the isomorphism $M^\si\cong W^\si\ot H\cong\hat{M}$.
Due to the universal property of cotensor coalgebras, we know that $\phi$ induces a coalgebra isomorphism $\Phi:T_H^c(M^\si)\cong T_H^c(\hat{M})$.
Using formula (\ref{mu}), we have $\Phi=\pi_0+\sum_{n\geq1}(\phi\circ\pi_1)^{\ot n}\De_c^{(n-1)}$, where $\pi_0,\pi_1$ are the projections to $H,M$ respectively.
\end{rem}

Next we define a QQSA structure on $T_H^c(M^\si)$ as follows.
Let $(f^\si,g^\si)$ be the corresponding pair on $T_H^c(M^\si)$ of quasi-symmetric type.
We only need to define $f^\si_{11}=\al_\si$. According to Proposition \ref{yd},
we can identify the $H$-Hopf bimodules $W^\si\ot H$ with $\hat{M}$ and define
\[\al_\si=\phi^{-1}\circ\al_{\hat{M}}\circ(\phi\ot\phi).\]
Similarly we can take the Hopf quotient $U_H(M^\si):=Q_H(M^\si)/J^\si$,
where $J^\si$ is the Hopf ideal of $Q_H(M^\si)$ generated by $\{\xi_i-K_i{K'_i}^{-1}+1:i\in I\}$.
It is easy to see that $\Phi$ becomes an isomorphism of Hopf algebras when we endow
$T_H^c(M^\si)$ with such QQSA structure. Moreover,
it naturally induces a Hopf algebra isomorphism $\bar{\Phi}$ between $U_H(M^\si)$ and $U_H(\hat{M})$, as $\xi_i-K_i{K'_i}^{-1}+1$ is fixed by $\Phi$.

Now we can inflate the group bicharacter $\si$ to be a Hopf 2-cocycle of $U_\q(\mg_A)$ as in \cite[Prop. 27]{PHR},
thus also obtain a Hopf 2-cocycle $\si$ of $U_H(M)$ due to Theorem \ref{iso}. That is, $\si$ satisfies
\[\begin{array}{l}
\si(\om_i,\om_j)\si^{-1}(\om_j,\om_i)=\hat{q}_{ij}q_{ij}^{-1},\\
\si(\om'_i,\om'_j)\si^{-1}(\om'_j,\om'_i)=\hat{q}_{ij}q_{ij}^{-1},\\
\si(\om'_i,\om_j)\si^{-1}(\om_j,\om'_i)=\hat{q}_{ji}^{-1}q_{ji},
\end{array}\]
for all $i, j\in I$ and gives $0$ for other homogeneous basis elements out of $U_\q(0)$.
It pushes us to consider the twist equivalence between multi-parameter quantum groups.
\begin{theorem}\label{tw}
Denote $\hat{\q}:=(\hat{q}_{ij})_{i,j\in I}$, which holds the condition (\ref{con}).
Then we have the following Hopf algebra isomorphism:
\[U_{\hat{\q}}(\mg_A)\cong U_\q(\mg_A)^\si,\]
where $U_\q(\mg_A)^\si$ is the Hopf 2-cocycle deformation of $U_\q(\mg_A)$.
In particular, if there exists $q\in\bK^\times$ such that $q_{ii}=q^{2d_i},~i\in I$,
we can take $\hat{q}_{ij}=q^{d_ia_{ij}}$ for all $i,j\in I$.
Then $U_\q(\mg_A)^\si$ is isomorphic to the one parameter version $U_{q,q^{-1}}(\mg_A)$,
proved in \cite[Theorem 28]{PHR} with another different Hopf $2$-cocycle.

\begin{proof}
Let $m^{\si}$ be the multiplication map of $U_\q(\mg_A)^\si$. Denote $a\di b:=m^{\si}(a,b)$ for $a, b\in U_\q(\mg_A)$. It
suffices to check the relations:
\begin{eqnarray*}
 &(R'1)&\quad \om_i^{\pm1}\di\om_j'^{\pm1}=\om_j'^{\pm1}\di\om_i^{\pm1},\quad
\om_i^{\pm1}\di\om_i^{\mp1}=\om_i'^{\pm1}\om_i'^{\mp1}=1,\\
 &(R'2)&\quad \om_i^{\pm1}\di\om_j^{\pm1}=\om_j^{\pm1}\di\om_i^{\pm1},
 \quad \om_i'^{\pm1}\di\om_j'^{\pm1}=\om_j'^{\pm1}\di\om_i'^{\pm1},\\
&(R'3)&\quad \om_i\di e_j\di\om_i^{-1}=\hat{q}_{ij} e_j,
\qquad \om_i'\di e_j \di\om_i'^{-1}=\hat{q}_{ji}^{-1} e_j,
\\
&(R'4)&\quad \om_i\di f_j\di\om_i^{-1}=\hat{q}_{ij}^{-1} f_j,\qquad
\om_i'\di f_j \di\om_i'^{-1}=\hat{q}_{ji} f_j,\\
&(R'5)&\quad e_i\di f_j-f_j\di e_i=\delta_{i,j}\frac{\hat{q}_{ii}}{\hat{q}_{ii}-1}({\om_i-\om_i'}),\\
&(R'6)&\quad \sum_{k=0}^{1-a_{ij}} (-1)^{k}
\binom{1-a_{ij}}{k}_{\hat{q}_{ii}}\hat{q}_{ii}^{\frac{k(k-1)}{2}}
\hat{q}_{ij}^{k}e_{i}^{\di(1-a_{ij}-k)}\di e_{j}\di e_{i}^{\di k} =0 \quad (i\ne j),\\
&(R'7)&\quad\sum_{k=0}^{1-a_{ij}} (-1)^{k}
\binom{1-a_{ij}}{k}_{\hat{q}_{ii}}\hat{q}_{ii}^{\frac{k(k-1)}{2}}
\hat{q}_{ij}^{k}f_{i}^{k}\di f_{j}\di f_{i}^{\di(1-a_{ij}-k)}=0 \quad (i\ne j).
\end{eqnarray*}
Since
\begin{gather*}
\De^2(\om_i)=\om_i\ot \om_i\ot \om_i,~\De^2(\om_i')=\om_i'\ot \om_i'\ot \om_i',\\
\De^2(e_i)=e_i\ot 1\ot1+\om_i\ot e_i\ot 1+\om_i\ot \om_i\ot e_i,\\
\De^2(f_i)= 1\ot1\ot f_i+1\ot f_i\ot \om_i'+f_i\ot \om_i'\ot \om_i'.
\end{gather*}
It is straightforward to check $(R'1)$ and $(R'2)$. For $(R'3)$
and $(R'4)$:
\begin{align*}
\om_i\di e_j&=\si(\om_i,\om_j)\om_ie_j=\si(\om_i,\om_j)q_{ij}e_j\om_i\\
&=\si(\om_i,\om_j)\si^{-1}(\om_j,\om_i)q_{ij}e_j\di\om_i=\hat{q}_{ij}e_j\di\om_i,
\end{align*}
\begin{align*}
\om'_i\di e_j&=\si(\om'_i,\om_j)\om'_ie_j=\si(\om'_i,\om_j)q_{ji}^{-1}e_j\om'_i\\
&=\si(\om'_i,\om_j)\si^{-1}(\om_j,\om'_i)q_{ji}^{-1}e_j\di\om'_i=\hat{q}_{ji}^{-1}e_j\di\om'_i,
\end{align*}
\begin{align*}
\om_i\di f_j&=\si^{-1}(\om_i,\om'_j)\om_if_j=\si^{-1}(\om_i,\om'_j)q_{ij}^{-1}f_j\om_i\\
&=\si(\om'_j,\om_i)\si^{-1}(\om_i,\om'_j)q_{ij}^{-1}f_j\di\om_i=\hat{q}_{ij}^{-1}f_j\di\om_i,
\end{align*}
\begin{align*}
\om'_i\di f_j&=\si^{-1}(\om'_i,\om'_j)\om'_if_j=\si^{-1}(\om'_i,\om'_j)q_{ji}f_j\om'_i\\
&=\si(\om'_j,\om'_i)\si^{-1}(\om'_i,\om'_j)q_{ji}f_j\di\om'_i=\hat{q}_{ji}f_j\di\om'_i,
\end{align*}
For $(R'5)$:
$$
e_i\di f_j-f_j\di e_i=e_if_j-f_je_i=\delta_{i,j}\frac{\hat{q}_{ii}}{\hat{q}_{ii}-1}(\om_i-\om_i').
$$
For $(R'6)$:
\begin{align*}
e_{i}^{\di(1-a_{ij}-k)}\di e_{j}\di e_{i}^{\di k}&=\prod_{r=0}^{-a_{ij}-k}\si(\om_i,\om_i^r\om_j)\prod_{r=0}^{k-1}
\si(\om_i^{1-a_{ij}-k+r}\om_j,\om_i)e_{i}^{1-a_{ij}-k} e_{j}e_{i}^{k}\\
&=\si(\om_i,\om_i)
^{\frac{-a_{ij}(1-a_{ij})}{2}}\si(\om_i,\om_j)^{1-a_{ij}-k}\si(\om_j,\om_i)^k
e_{i}^{1-a_{ij}-k} e_{j}e_{i}^{k}\\
&=\si(\om_i,\om_i)
^{\frac{-a_{ij}(1-a_{ij})}{2}}\si(\om_i,\om_j)^{1-a_{ij}}\hat{q}_{ji}^kq_{ji}^{-k}
e_{i}^{1-a_{ij}-k} e_{j}e_{i}^{k}.
\end{align*}
Hence
\begin{eqnarray*}
&&\sum_{k=0}^{1-a_{ij}} (-1)^{k}
\binom{1-a_{ij}}{k}_{\hat{q}_{ii}}\hat{q}_{ii}^{\frac{k(k-1)}{2}}\hat{q}_{ij}^k
e_{i}^{\di(1-a_{ij}-k)}\di e_{j}\di e_{i}^{\di k} \\
&=&\si(\om_i,\om_i)
^{\frac{-a_{ij}(1-a_{ij})}{2}}\si(\om_i,\om_j)^{1-a_{ij}}\sum_{k=0}^{1-a_{ij}} (-1)^{k}
\binom{1-a_{ij}}{k}_{\hat{q}_{ii}}\hat{q}_{ii}^{\frac{k(k-1)}{2}}\hat{q}_{ij}^k\hat{q}_{ji}^kq_{ji}^{-k}
e_{i}^{1-a_{ij}-k}e_{j}e_{i}^k\\
&=&\si(\om_i,\om_i)
^{\frac{-a_{ij}(1-a_{ij})}{2}}\si(\om_i,\om_j)^{1-a_{ij}}\sum_{k=0}^{1-a_{ij}} (-1)^{k}
\binom{1-a_{ij}}{k}_{q_{ii}}q_{ii}^{\frac{k(k-1)}{2}}q_{ij}^k
e_{i}^{1-a_{ij}-k}e_{j}e_{i}^k=0.
\end{eqnarray*}
For $(R'7)$:
\begin{align*}
f_{i}^{\di k}\di f_{j}\di f_{i}^{\di(1-a_{ij}-k)}&=\prod_{r=0}^{k-1}\si^{-1}(\om'_i,{\om'_i}^r\om'_j)\prod_{r=0}^{-a_{ij}-k}
\si^{-1}({\om'_i}^{k+r}\om'_j,\om'_i)f_{i}^kf_{j}f_{i}^{1-a_{ij}-k}\\
&=\si(\om'_i,\om'_i)
^{\frac{a_{ij}(1-a_{ij})}{2}}\si(\om'_i,\om'_j)^{-k}\si(\om'_j,\om'_i)^{-(1-a_{ij}-k)}
f_{i}^kf_{j}f_{i}^{1-a_{ij}-k}\\
&=\si(\om'_i,\om'_i)
^{\frac{a_{ij}(1-a_{ij})}{2}}\si(\om'_j,\om'_i)^{-(1-a_{ij})}\hat{q}_{ji}^kq_{ji}^{-k}
f_{i}^kf_{j}f_{i}^{1-a_{ij}-k}.
\end{align*}
Hence
\begin{eqnarray*}
&&\sum_{k=0}^{1-a_{ij}} (-1)^{k}
\binom{1-a_{ij}}{k}_{\hat{q}_{ii}}\hat{q}_{ii}^{\frac{k(k-1)}{2}}\hat{q}_{ij}^k
f_{i}^{\di k}\di f_{j}\di f_{i}^{\di(1-a_{ij}-k)}\\
&=&\si(\om'_i,\om'_i)
^{\frac{a_{ij}(1-a_{ij})}{2}}\si(\om'_j,\om'_i)^{-(1-a_{ij})}\sum_{k=0}^{1-a_{ij}} (-1)^{k}
\binom{1-a_{ij}}{k}_{\hat{q}_{ii}}\hat{q}_{ii}^{\frac{k(k-1)}{2}}\hat{q}_{ij}^k\hat{q}_{ji}^kq_{ji}^{-k}
f_{i}^kf_{j}f_{i}^{1-a_{ij}-k}\\
&=&\si(\om'_i,\om'_i)
^{\frac{a_{ij}(1-a_{ij})}{2}}\si(\om'_j,\om'_i)^{-(1-a_{ij})}\sum_{k=0}^{1-a_{ij}} (-1)^{k}
\binom{1-a_{ij}}{k}_{q_{ii}}q_{ii}^{\frac{k(k-1)}{2}}q_{ij}^k
f_{i}^kf_{j}f_{i}^{1-a_{ij}-k}=0.
\end{eqnarray*}
The proof is complete.
\end{proof}
\end{theorem}

On the other hand, the group bicharacter $\si$ induces a Hopf $2$-cocycle of $T_H^c(M)$,
also denoted by $\si$, via the projection $\pi_0$ of $T_H^c(M)$ to $H$.
Here we compare $T_H^c(M^\si)$ with $T_H^c(M)^\si$ as follows.
\begin{prop}\label{id}
The identity map on the underlying space $T_H^c(M)$ provides a Hopf algebra
isomorphism between $T_H^c(M^\si)$ and $T_H^c(M)^\si$, if $\si(K_i,K'_i)=1$ for all $i\in I$.
\end{prop}
\begin{proof}
First we note that the degree one component $T_H^c(M)^\si_1$ is isomorphic to $M^\si$ as $H$-Hopf bimodules:
\[\begin{split}
h*_\si m&=\si(h_{(1)},m_{(-1)})\si^{-1}(h_{(3)},m_{(1)})h_{(2)}*m_{(0)}\\
&=\si(h_{(1)},m_{(-1)})\si^{-1}(h_{(3)},m_{(1)})h_{(2)}\cdot m_{(0)}=h\cdot_\si m,
\end{split}
\]
and $m*_\si h=m\cdot_\si h$ similarly for all $h\in H,m\in M$.
Meanwhile, we show that $\al_\si=\si\star\al\star\si^{-1}$, where $\star$ is the convolution product on $\mb{Hom}(T_H^c(M)^{\ot2}, T_H^c(M))$.
In fact, for any monomials $K, K'\in H$,
\[\begin{array}{l}
\begin{split}
\al_\si(E_iK\ot F_jK')&=\phi^{-1}\circ\al_{\hat{M}}\circ(\phi\ot\phi)(E_iK\ot F_jK')\\
&=\si(K_i,K^{-1})\si({K'_j}^{-1},K'^{-1})\si(K_i{K'_i}^{-1},KK')\al_{\hat{M}}(E_iK\ot F_jK')\\
&=\si(K_i,K')\si(K,{K'_i}^{-1})\al(E_iK\ot F_jK'),
\end{split}\\
\begin{split}
\si\star\al\star\si^{-1}(E_iK\ot F_jK')&=\si(K_iK,{K'_j}^{-1}K')\si^{-1}(K,K')\al(E_iK\ot F_jK')\\
&=\si(K_i,{K'_i}^{-1})\si(K_i,K')\si(K,{K'_i}^{-1})\al(E_iK\ot F_jK').
\end{split}
\end{array}
\]
As $\si(K_i,K'_i)=1$ for all $i\in I$, we know that both sides are equal. For other bases, both sides give $0$.
Now let $*'$ be the multiplication of $T_H^c(M^\si)$. By the formula (\ref{mu}) and the above computation, we have
\[
*'=g^\si+\sum_{n\geq1}{f^{\si}}^{\ot n}\circ \De^{(n-1)}=\si\star\lb g+\sum_{n\geq1}f^{\ot n}\circ \De^{(n-1)}\rb \star\si^{-1}=\si\star *\star\si^{-1}.
\]
Hence, we can identify the multiplications of $T_H^c(M^\si)$ and $T_H^c(M)^\si$.
\end{proof}
When $\si$ also satisfies the extra condition in Proposition \ref{id}, we can identify $Q_H(M^\si)$ with $Q_H(M)^\si$. Abusing the notation $J^\si$ to denote the  Hopf ideal of $Q_H(M)^\si$ parallel to that of $Q_H(M^\si)$, we get the Hopf quotient $U^\si:=Q_H(M)^\si/J^\si$, naturally isomorphic to $U_H(M^\si)$. Combining it with Theorems \ref{iso}, \ref{tw}, we have the following commutative diagram of Hopf algebras:
\[\xymatrix@=2em{U_\q(\mg_A)^\si\ar@{->}[r]^-{\sim}\ar@{->}[d]^-{\Psi}&U_{\hat{\q}}(\mg_A)\ar@{->}[d]^-{\Psi}\\
U_H(M)^\si\ar@{->}[r]^-{\sim}&U_H(\hat{M})\\
U^\si\ar@{->}[r]^-{\sim}\ar@{->}[u]_-{\bar{\Phi}}&U_H(M^\si)\ar@{->}[u]_-{\bar{\Phi}}}\]
where all the horizontal isomorphisms are natural identifications.

\section{Quantum quasi-symmetric construction of irreducible representations}

\noindent{5.1.}
In this section, we use the coinvariant subspace of the degree one component in the QQSAs to realize the irreducible module of multi-parameter quantum groups. First we recall the representation theory of $U_{\q}$ described in \cite{PHR}.

\begin{defn}
The category $\mO^{\q}$  consists of $U_{\q}(\mg_A)$-modules
$V^{\q}$ with the following conditions satisfied:

$(1)$\ $V^{\q}$ has a weight space decomposition
$V^{\q}=\bigoplus_{\la\in\La}V^{\q}_{\la}$, where
$$
V^{\q}_{\la}=\{v\in V^{\q}\mid\om_{i}v=q_{\al_i\la}v,\
\om_{i}'v=q_{\la\al_i}^{-1}v, \ \forall\; i\in I\}
$$
and $\dim V_{\la}^{\q}<\infty$ for all $\la\in\La$.

$(2)$\ There exist a finite number of elements $\la_1,\dots,
\la_t\in \La$ such that the weight set
$$
\wt (V^{\q})\subset D(\la_1)\cup\cdots\cup D(\la_t),
$$
where $D(\la_i):=\{\mu\in\La\,|\,\mu<\la_i\}$.

$(3)$\  $e_i, i\in I$ are locally nilpotent on $V^{\q}$.

The morphisms are taken to be usual $U_{\q}(\mg_A)$-module
homomorphisms. If $f_i, \,i\in I$ are also locally nilpotent, we get a subcategory $\mO_{int}^{\q}$ consisting of integrable modules.
\end{defn}

For any $\la\in\La$ and $i\in I$, write $\lan\la,\al_i^\vee\ran:=\dfrac{2(\la,\al_i)}{(\al_i,\al_i)}$ as usual, then it is easy to see that
\begin{equation}\label{wei}
q_{\al_i\la}q_{\la\al_i}=q_{ii}^{\lan\la,\al_i^\vee\ran}.
\end{equation}

\begin{prop}[{\cite[Prop. 41, 58]{PHR}}]\label{iteg}
For the generic case, let $V^{\q}(\la)$ be an irreducible highest weight
module with highest weight vector $v_{\la}$. Then $V^{\q}(\la)$ belongs to category $\mO^{\q}_{int}$ if and only if $\la\in \La^+$.
Moreover, $\mO^{\q}_{int}$ is a semisimple category in the generic case, $V^{\q}(\la)\cong U_{\q}/J$ as $U_{\q}$-modules,
where $J$ is the left ideal of $U_{\q}$ generated by \[e_i,\quad \om_i-q_{\al_i\la}\cdot1,\quad\om'_i-q_{\la\al_i}^{-1}\cdot1,\quad f_i^{1+\lan\la,\al_i^\vee\ran},\qquad\forall\; i\in I.\]
\end{prop}

\noindent{5.2.}
In order to realize the irreducible module $V^{\q}(\la)$ for $\la\in \La^+$,
we recall the machinery constructed by Radford in \cite{Rad}.
One can also check this material in \cite[\S1.5]{AuS}.
\begin{defn}
Let $H$ and $K$ be two Hopf algebras. The pair $(H,K)$ is called a
\textit{Radford pair} if there exists Hopf algebra morphisms $i: H\rw K$ and $p: K\rw H$ such that $p\circ i=\mi_H$.
\end{defn}

Once a Radford pair $(H,K)$ is given, we have the following isomorphism of Hopf algebras $\Phi:K\cong R\# H,~a\mapsto \sum a_{(1)}(i\circ p \circ S)(a_{(2)})\ot p(a_{(3)})$, where $R=K^{\bs{co}H}=\{k\in K\mid (\mi\ot p)\De(k)=k\ot 1\}$, the set of right coinvariants,
is a braided Hopf algebra in ${^H_H}\m{Y}\m{D}$.

Now fix a weight $\la\in\La^+$, and define $T=\bK[K_i^{\pm1},{K'_i}^{\pm1},K_\la^{\pm1}:i\in I]$
as the group algebra of the free abelian group $\bZ^{2|I|+1}$, containing $H$ defined in Section 3 as a Hopf subalgebra.
Meanwhile, we take the enlarged vector spaces $W'=W\op\bK v_\la$ and $N=W'\ot T$.
Analogous to the former construction, we consider the following $T$-Hopf bimodule structure on $N$:

(1) $N$ has the trivial right $T$-Hopf module structure.

(2) The left $T$-Hopf module structure of $N$ is defined as follows: Similarly we first deal with $W'$.
All the module and comodule structures defined on $W$ are preserved for $W'$ when restricted to $W$.
Here complete the remaining cases.
\[\begin{array}{l}
K_\la\cdot E_i=q_{\al_i\la}^{-1}E_i,\quad K_\la\cdot F_i=q_{\al_i\la}F_i,\quad K_\la\cdot \xi_i=\xi_i,\\
K_i\cdot v_\la=q_{\al_i\la}v_\la,\quad K'_i\cdot v_\la=q_{\la\al_i}^{-1}v_\la,\quad K_\la\cdot v_\la=q_{\la\la}v_\la,\\
\de_L(v_\la)=K_\la\ot v_\la.
\end{array}
\]
Meanwhile, $T$ derives the left $T$-module and comodule structure from the left multiplication and its comultiplication.
$N=W'\ot T$ becomes a left $T$-Hopf module as the tensor product. Moreover, we define the map $\al_N: N\ot N\rw N$ as the trivial inflation of $\al_M$ on $N\ot N$ as follows:
For any $K, K'\in T$, if $\mu$ is the constant such that $K\cdot F_j=\mu F_j$, then
\[\al_N(E_iK\ot F_jK')=\de_{ij}\dfrac{q_{ii}\mu\xi_iKK'}{q_{ii}-1},~i,j\in I,\]
and for any other elements not of the above form, $\al_N$ gives $0$.

Similarly, the above setting also provides a quantum quasi-symmetric algebra $Q_T(N)$.
For instance, $\al_N$ again satisfies the condition (3) given in Definition \ref{qsa}, as
\[\begin{split}
\al_N(K_\la\cdot(E_iK\ot F_jK'))&=\al_N((K_\la\cdot E_i)K_\la K\ot F_jK')=\de_{ij}\dfrac{q_{\al_i\la}^{-1}q_{ii}q_{\al_j\la}\mu\xi_iK_\la KK'}{q_{ii}-1}\\
&=\de_{ij}\dfrac{q_{ii}\mu\xi_iK_\la KK'}{q_{ii}-1}=K_\la\cdot\al_N(E_iK\ot F_jK'),
\end{split}\]

The ideal of $Q_T(N)$ generated by $\xi_i-K_i{K'_i}^{-1}+1,~i\in I$ is also a Hopf ideal.
Thus, we have the corresponding Hopf quotient $U_T(N)$.
It is clear that $Q_T(M)$ is a Hopf subalgebra of $Q_T(N)$.
We define a gradation on $Q_T(N)$ by setting $\mb{deg}(v_\la)=1$ and all elements in $T$ and $M$ be of degree $0$.
Then we get a graded Hopf algebra
\[Q_T(N)=\bigoplus_{n\in\bN}Q_T(N)_{(n)}\]
with $Q_T(N)_{(0)}=Q_T(M)$. Let $i: Q_T(M)\rw Q_T(N)$ and $p:Q_T(N)\rw Q_T(M)$
be the embedding into degree $0$ and the projection onto degree $0$ respectively.
One gets a Radford pair $(Q_T(M),Q_T(N))$. Since the Hopf ideal $\lb\xi_i-K_i{K'_i}^{-1}+1\mid i\in I\rb$ is homogeneous,
both $U_T(M)$ and $U_T(N)$ inherit the gradation from $Q_T(M)$ and $Q_T(N)$ such that $U_T(N)_{(0)}=U_T(M)$.
That is another Radford pair $(U_T(M),U_T(N))$.

Now applying the machinery of Radford, we have an isomorphism of $U_T(M)$-Hopf bimodules $U_T(N)\cong U_T(N)^{\bs{coR}}\ot U_T(M)$, where $U_T(N)^{\bs{coR}}$ admits a $U_T(M)$-Yetter-Drinfel'd module structure via the adjoint action.
Moreover, it can focus on any homogeneous components,
\[U_T(N)_{(n)}\cong U_T(N)_{(n)}^{\bs{coR}}\ot U_T(M).\]
In particular, we highlight the degree one component $R(1):=U_T(N)_{(1)}^{\bs{coR}}$.

In Theorem \ref{iso}, we give a Hopf algebra isomorphism $\psi:U_{\q}(\mg_A)\rw U_H(M)$.
Now we show that $R(1)$ realizes the irreducible module $V^{\q}(\la)$ via the map $\psi$.

\begin{theorem}
For the generic case, $R(1)\cong V^{\q}(\la)$ as left $U_{\q}(\mg_A)$-modules.
\end{theorem}
\begin{proof}
Using the isomorphism $\Phi$ from the Radford pair $(U_T(M),U_T(N))$, one can see that $R(1)=\mb{ad}_l(U_T(M))(v_\la)=\mb{ad}_l(U_H(M))(v_\la)$.
Indeed, any element in $R(1)$ looks like $\sum_j x_j*v_\la*y_j$ with $x_j,y_j\in U_T(M)$, then
\[\begin{split}
\sum_j&x_j*v_\la*y_j=\sum_j\lb (x_j)_{(1)}*v_\la*(y_j)_{(1)}\rb*S\lb (x_j)_{(2)}*(y_j)_{(2)}\rb\\
&=\sum_j(x_j)_{(1)}*v_\la*S\lb (x_j)_{(2)}\rb=\mb{ad}_l(\sum_jx_j)(v_\la)\in\mb{ad}_l(U_T(M))(v_\la).
\end{split}\]
Now we check that the ideal $J$ given in Proposition \ref{iteg} annihilates the vector $v_\la$. For any $i\in I$,
\[\begin{array}{l}
\begin{split}
e_i\cdot v_\la&=\mb{ad}_l(\psi(e_i))(v_\la)=E_i*v_\la-K_i*v_\la*K_i^{-1}*E_i
=E_i*v_\la-q_{\al_i\la}v_\la*E_i\\
&=E_iK_\la\ot v_\la+(K_i\cdot v_\la)K_i\ot E_i-q_{\al_i\la}(v_\la K_i\ot E_i+
(K_\la\cdot E_i)K_\la\ot v_\la)=0,
\end{split}\\
\om_i\cdot v_\la=\mb{ad}_l(\psi(\om_i))(v_\la)=K_i*v_\la*K_i^{-1}=q_{\al_i\la}v_\la\\
\om'_i\cdot v_\la=\mb{ad}_l(\psi(\om'_i))(v_\la)=K'_i*v_\la*{K'_i}^{-1}=q_{\la\al_i}^{-1}v_\la\\
\end{array}
\]
For $f_i^{1+\lan\la,\al_i^\vee\ran}\cdot v_\la$, we have
\[\begin{split}
f_i\cdot v_\la&=\mb{ad}_l(\psi(f_i))(v_\la)=F_iK'_i*v_\la*{K'_i}^{-1}-v_\la*F_i
=q_{\la\al_i}^{-1}F_i*v_\la-v_\la*F_i\\
&=q_{\la\al_i}^{-1}\lb F_iK_\la\ot v_\la+({K'_i}^{-1}\cdot v_\la){K'_i}^{-1}\ot F_i\rb-\lb v_\la{K'_i}^{-1}\ot F_i+
(K_\la\cdot F_i)K_\la\ot v_\la\rb\\
&=(q_{\la\al_i}^{-1}-q_{\al_i\la})F_iK_\la\ot v_\la.
\end{split}
\]
When $r\geq2$, we have
\[\begin{split}
f_i\cdot&\lb F_i{K'_i}^{-(r-2)}K_\la\ot\cdots\ot F_i{K'_i}^{-1}K_\la\ot F_iK_\la\ot v_\la\rb\\
&=\mb{ad}_l(\psi(f_i))\lb F_i{K'_i}^{-(r-2)}K_\la\ot\cdots\ot F_i{K'_i}^{-1}K_\la\ot F_iK_\la\ot v_\la\rb\\
&=F_iK'_i*\lb F_i{K'_i}^{-(r-2)}K_\la\ot\cdots\ot F_i{K'_i}^{-1}K_\la\ot F_iK_\la\ot v_\la\rb*{K'_i}^{-1}\\
&-\lb F_i{K'_i}^{-(r-2)}K_\la\ot\cdots\ot F_i{K'_i}^{-1}K_\la\ot F_iK_\la\ot v_\la\rb*F_i\\
&=q_{ii}^{r-1}q_{\la\al_i}^{-1}F_i*\lb F_i{K'_i}^{-(r-2)}K_\la\ot\cdots\ot F_i{K'_i}^{-1}K_\la\ot F_iK_\la\ot v_\la\rb\\
&-\lb F_i{K'_i}^{-(r-2)}K_\la\ot\cdots\ot F_i{K'_i}^{-1}K_\la\ot F_iK_\la\ot v_\la\rb*F_i\\
&=(r)_{q_{ii}^{-1}}(q_{ii}^{r-1}q_{\la\al_i}^{-1}-q_{\al_i\la}) F_i{K'_i}^{-(r-1)}K_\la\ot\cdots\ot F_i{K'_i}^{-1}K_\la\ot F_iK_\la\ot v_\la.
\end{split}
\]
Hence, $f_i^r\cdot v_\la=(r)_{q_{ii}^{-1}}!\prod\limits_{k=1}^r(q_{ii}^{k-1}q_{\la\al_i}^{-1}-q_{\al_i\la})F_i{K'_i}^{-(r-1)}K_\la\ot\cdots\ot F_i{K'_i}^{-1}K_\la\ot F_iK_\la\ot v_\la$ for any $r\geq1$.
In particular, $f_i^{1+\lan\la,\al_i^\vee\ran}\cdot v_\la=0$ for all $i\in I$.

As a consequence, we get a left $U_{\q}$-module surjection $V^{\q}(\la)\cong U_{\q}/J \twoheadrightarrow R(1)$ with $1$ mapping to $v_\la$.
Since $V^{\q}(\la)$ is irreducible, that is an isomorphism.
\end{proof}

\begin{rem}
Now suppose that $q_{ii}~(i\in I)$ are all roots of unity with $\mb{ord}(q_{ii})=\ell_i$,
and $A=(a_{ij})_{i,j\in I}$ is a finite Cartan matrix. If for any $i\in I$, $\ell_i$ is odd and
not divisible by $3$ if $i$ belongs to a connected component of type $G_2$,
then we have $\ell_i=\ell_j$ whenever $i, j$ lie in the same connected component \cite[\S4]{AS1}.

It is harmless to assume further that $A$ is indecomposable, i.e., the corresponding Dynkin diagram is connected,
then we can set $\ell=\ell_i$ for all $i\in I$ under the previous assumption.
Now it can be seen that if the weight $\la\in\La^+$ lies in the fundamental alcove, i.e.,
$0<\lan\la+\rho,\al^\vee\ran<\ell$ for all $\al\in\Phi^+$,
then the module $R(1)$ can still realize the irreducible module $V^\q(\la)$.
%
\end{rem}


\medskip
\bibliographystyle{amsalpha}

\end{document}